\documentclass[12pt,twoside]{article}
\usepackage{amsmath,amsfonts,amssymb,a4,enumitem,epsfig,array,psfrag}
\usepackage{url}
\usepackage{amsthm}
\usepackage{subfig}
\usepackage{chngcntr} 
\usepackage{color}
\usepackage[section]{placeins}
\usepackage[colorlinks=true]{hyperref}

\usepackage{algorithm}
\usepackage{algpseudocode}

\graphicspath{{figures/}}

\newtheorem{prop}{Proposition}

\newtheorem{lemma}[prop]{Lemma}

\theoremstyle{definition}

\newtheorem{rem}[prop]{Remark}

\newcommand{\half}{{\frac{1}{2}}}

\newcommand{\ZZ}{{\mathbb{Z}}}

\newcommand{\RR}{{\mathbb{R}}}

\newcommand{\interior}{\operatorname{int}}
\newcommand{\cR}{\mathcal{R}}
\newcommand{\cD}{\mathcal{D}}
\newcommand{\cS}{\mathcal{S}}

\newcommand{\Tw}{\operatorname{Tw}}

\newcommand{\ccol}{\operatorname{color}} 

\newcommand{\charge}{\operatorname{charge}}
\newcommand{\wind}{\operatorname{wind}}

\newcommand{\metricw}[2]{\operatorname{w}_{\operatorname{metric}}(#1,#2)}
\newcommand{\topw}[2]{\operatorname{w}_{\operatorname{top}}(#1,#2)}

\newcommand{\tv}{\vec{v}}
\newcommand{\vu}{\vec{u}}

\newcommand{\ex}{\vec{\mathbf{i}}}
\newcommand{\ey}{\vec{\mathbf{j}}}
\newcommand{\ez}{\vec{\mathbf{k}}}

\newcommand{\neighbor}{\mathcal{N}}

\newcommand{\myScaleVar}{1} 

\newcommand{\plshalf}[1]{{#1}^\sharp}

\makeatletter
\newcommand{\subjclass}[2][2010]{%
  \let\@oldtitle\@title%
  \gdef\@title{\@oldtitle\footnotetext{#1 \emph{Mathematics Subject Classification.} #2}}%
}
\newcommand{\keywords}[1]{%
  \let\@@oldtitle\@title%
  \gdef\@title{\@@oldtitle\footnotetext{\emph{Key words and phrases.} #1.}}%
}
\makeatother

\date{November 3, 2014}

\begin{document}
\title{Twists for duplex regions}
\author{Pedro H. Milet \and Nicolau C. Saldanha} 

\maketitle
\begin{abstract}
This note relies heavily on \url{arXiv:1404.6509} and \url{arXiv:1410.7693}. Both articles discuss domino tilings of three-dimensional regions, and both are concerned with \emph{flips}, the local move performed by removing two parallel dominoes and placing them back in the only other possible position.
In the second article, an integer $\Tw(t)$ is defined for any tiling $t$ of a large class of regions $\cR$: it turns out that $\Tw(t)$ is invariant by flips. In the first article, a more complicated polynomial invariant $P_t(q)$ is introduced for tilings of two-story regions. It turns out that $\Tw(t) = P_t'(1)$ whenever $t$ is a tiling of a \emph{duplex region}, a special kind of two-story region for which both invariants are defined. This identity is proved in \url{arXiv:1410.7693} in an indirect and nonconstructive manner. In the present note, we provide an alternative, more direct proof.  
\end{abstract}
\section{Introduction}
We assume that the reader is familiar to the notations of \cite{primeiroartigo, segundoartigo}. In particular, a \emph{multiplex region with axis} $\ez$ (or $\ez$-\emph{multiplex}) is a region of the form $\cD + [0,N]\ez$, where $\cD \subset \RR^2 \times \{0\}$ is simply connected and has connected interior; if $N = 2$, the region is a ($\ez$-)\emph{duplex region}. The twist $\Tw(t)$ is an integer defined below for tilings of a multiplex, and the polynomial $P_t(q) \in \ZZ[q,q^{-1}]$ is defined for tilings of a duplex region via the formula \eqref{eq:definitionOfPt} below.


Recall that given $\vu \in \Phi = \{\pm \ex, \pm \ey, \pm \ez\}$, we define the $\vu$-\emph{pretwist} of a tiling $t$ of a region $\cR$ as 
$
T^{\vu}(t) = \sum_{d_0,d_1 \in t} \tau^{\vu}(d_0,d_1),
$
where $\tau^{\vu}$ denotes the effect along $\vu$, or 
$$\tau^{\vu}(d_0,d_1) = \begin{cases} \frac{1}{4} \det(\tv(d_1),\tv(d_0),\vu), &d_1 \cap \cS^{\vu}(d_0) \neq \emptyset \\ 0, &\mbox{otherwise.} \end{cases}$$
Here $\cS^{\vu}(d_0)$ denotes the open $\vu$-shade of $d_0$, as defined in Section 3 of \cite{segundoartigo}, and $\tv(d) \in \Phi$ denotes the center of the black cube contained in $d$ minus the center of the white one.

\begin{prop}
\label{prop:duplexes}
If $\cR$ is a duplex region, then, for any tiling $t$ of $\cR$,
$$P_t'(1) = T^{\ex}(t) = T^{\ey}(t) = T^{\ez}(t).$$ 
\end{prop}
The equality $T^{\ex}(t) = T^{\ey}(t) = T^{\ez}(t)$ above is a special case of Proposition 3.3 in \cite{segundoartigo}, and this value is, by definition, the twist $\Tw(t)$. In this note, we give an independent proof of this equality in the particular case of duplex regions. 
In the aforementioned article, we present a different, shorter proof of the equality $P_t'(1) = \Tw(t)$ using the connectivity of the space of domino tilings of a duplex region by flips and trits. The proof here presented is longer, but more direct.  

The authors gratefully acknowledge the support from FAPERJ, CNPq and CAPES.
\section{Socks and winding numbers}
Let $\cR$ be a duplex region. Consider the undirected plane graph $G$ whose vertex set is 
$$\left\{\left(\lfloor x \rfloor, \lfloor y \rfloor \right) : (\plshalf{x},\plshalf{y},\plshalf{z}) \text{ is the center of a cube in } R \right\},$$
and where two vertices are joined by an edge if their Euclidean distance is $1$. A \emph{system of cycles}, or \emph{sock}, in $G$ is a directed subgraph of $G$ consisting solely of oriented simple cycles. A \emph{jewel} of a sock is a vertex of $G$ that is not contained in the sock. A vertex $v = (x,y) \in \ZZ^2$ is called \emph{white} (resp. \emph{black}) if $x + y$ is even (resp. \emph{odd}). We set $\ccol(v) = 1$ if $v$ is black, and $-1$ if it is white.

Each tiling $t$ of $\cR$ has a unique corresponding sock in $G$, where trivial cycles in the associated drawing of $t$ are represented as pairs of adjacent jewels: this is illustrated in Figure \ref{fig:tilingTwoFloors_withSOC}. Each sock may refer to a set of tilings, all in the same flip connected component (therefore, with the same $P_t'(1)$ and the same $\Tw(t)$).

\begin{figure}[ht]%
\centering
\includegraphics[width=0.6\columnwidth]{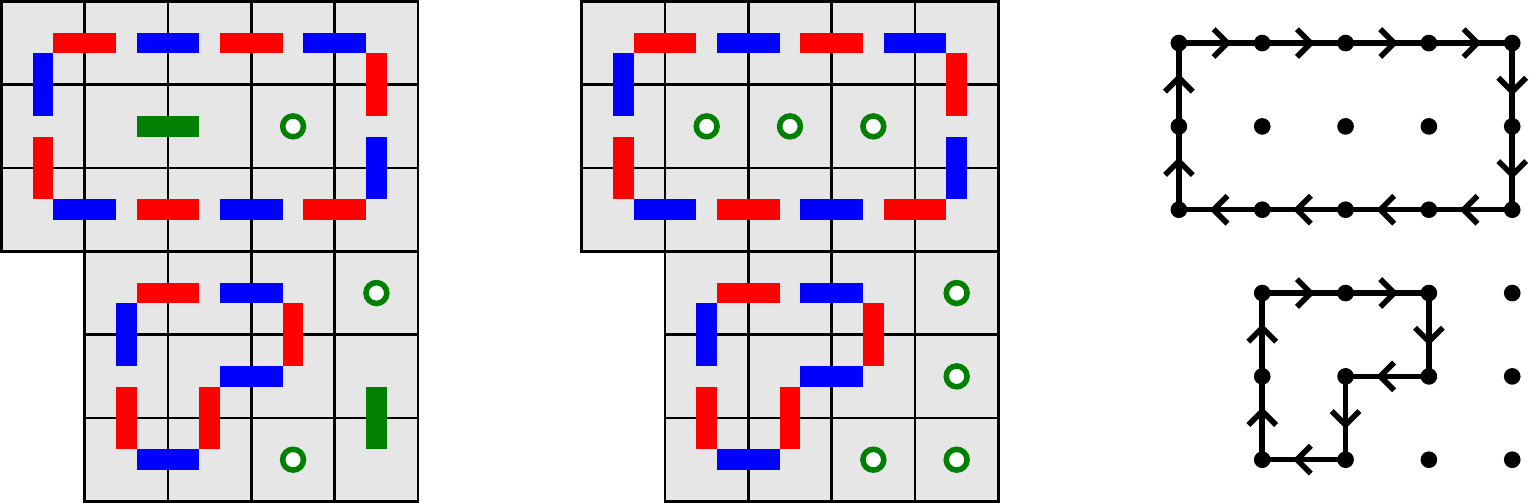}%
\caption{A tiling with two trivial cycles; the same tiling with the two trivial cycles flipped into jewels; and the sock that corresponds to both of them.}%
\label{fig:tilingTwoFloors_withSOC}%
\end{figure}

Let $t$ be a tiling of $\cR$, and let $s$ be its corresponding sock in $G$. For $p \in \RR^2$ and a cycle $\gamma$ of $s$, let $\wind(\gamma,p)$ be the winding number of $\gamma$, thought of as a curve in $\RR^2$, around $p$. Clearly we can write our invariant $P_t(q)$ as
\begin{equation}
P_t(q) = \sum_{v \in \ZZ^2} \ccol(v) q^{\sum_{\gamma, v \notin \gamma} \wind(\gamma,v)},
\label{eq:definitionOfPt}
\end{equation}
where the sum in the exponent of $q$ is taken over all the cycles in $s$ that do not contain $v$.

\begin{lemma}
\label{lemma:pretwistsNotAxis}
If $\cR$ is a $\ez$-duplex region with associated graph $G$, $\vu \in \{\pm \ex, \pm \ey\}$ and $t$ is a tiling of $\cR$ with corresponding sock $s$, then
$T^{\vu}(t) = P_t'(1)$.
\end{lemma} 
\begin{proof}
Two dominoes that are not parallel to $\ez$ have no effect along $\vu$ on one another. Therefore, we only consider pairs of dominoes where one is parallel to $\ez$, that is, refers to a jewel of $s$.

If $\gamma$ is a cycle of $s$ and $v$ is a jewel, one way of computing $\wind(\gamma,v)$ is to count (with signs) the intersections of $\gamma$ with the half-line $v +  [0,\infty) \vu$. Thus, if $d_v$ denotes the domino containing $v$ and $d \in \gamma$ means that $d$ refers to an edge of $\gamma$, then
$\ccol(v)\wind(\gamma,v) = 2 \sum_{d \in \gamma} \tau^{\vu}(d,d_v) = 2 \sum_{d \in \gamma} \tau^{\vu}(d_v,d).$
Thus, $$P_t'(1) = \sum_{\gamma, v} \ccol(v)\wind(\gamma,v) = \sum_{\substack{\gamma,v\\d \in \gamma}} (\tau^{\vu}(d,d_v)+ \tau^{\vu}(d_v,d)) = T^{\vu}(t),$$
completing the proof.
\end{proof}
\section{Charges and weights}
We now consider $T^{\ez}$. Again, let $t$ be a tiling of a duplex region with corresponding sock $s$.
Let the \emph{charge enclosed} by a cycle $\gamma$ of $s$ be $$\charge_{\interior}(\gamma) = \sum_{v \notin \gamma} \ccol(v)\wind(\gamma,v),$$
so that  
$P_t'(1) = \sum_{\gamma \scalebox{0.7}{\mbox{ cycle of }} s} \charge_{\interior}(\gamma).$
Charges can be looked at from a point of view that is more interesting for our purposes.  Given $v \in \RR^2$, consider the set of four points $\neighbor_v = \{v + (\frac{k}{2}, \frac{l}{2}) | k,l \in \{-1, 1\}\},$ i.e., the set of points of the form $v + (\pm \half, \pm \half)$. The \emph{metric weight} of a vertex $v \in \ZZ^2$ with respect to a cycle $\gamma$ of $s$ is given by 
$$ \metricw{\gamma}{v} = \frac{1}{4} \sum_{u \in \neighbor_v} \wind(\gamma,u),$$
while the \emph{topological weight} $\topw{\gamma}{v}$ of $v$ is the (arithmetic) average of the set $\wind(\gamma,\neighbor_v) = \{\wind(\gamma,u) | u \in \neighbor_v\}$ (see Figure \ref{fig:topAndMetricWeights}). 

\begin{lemma}
\label{lemma:interiorCharge}
$$\charge_{\interior}(\gamma) = \sum_{v \in \ZZ^2} \ccol(v)\topw{\gamma}{v}.$$
\end{lemma}
\begin{proof}
Notice that
$$ \topw{\gamma}{v} = \begin{cases} \wind(\gamma,v), &\text{if } v \notin \gamma,\\  
 \half, &\text{if } v \in \gamma \text{ and $\gamma$ is counterclockwise oriented,}\\
-\half, &\text{if } v \in \gamma \text{ and $\gamma$ is clockwise oriented.}\end{cases}$$
In particular, $\sum_{v \in \gamma}\topw{\gamma}{v} \ccol(v) = \pm \half \sum_{v \in \gamma}\ccol(v) = 0.$ Hence, 
$$\charge_{\interior}(\gamma) = \sum_{v \in \ZZ^2} \ccol(v)\topw{\gamma}{v}.$$
\end{proof}

\begin{figure}%
\centering
\def\svgwidth{0.7\columnwidth}
\def\myScaleVar{0.8}
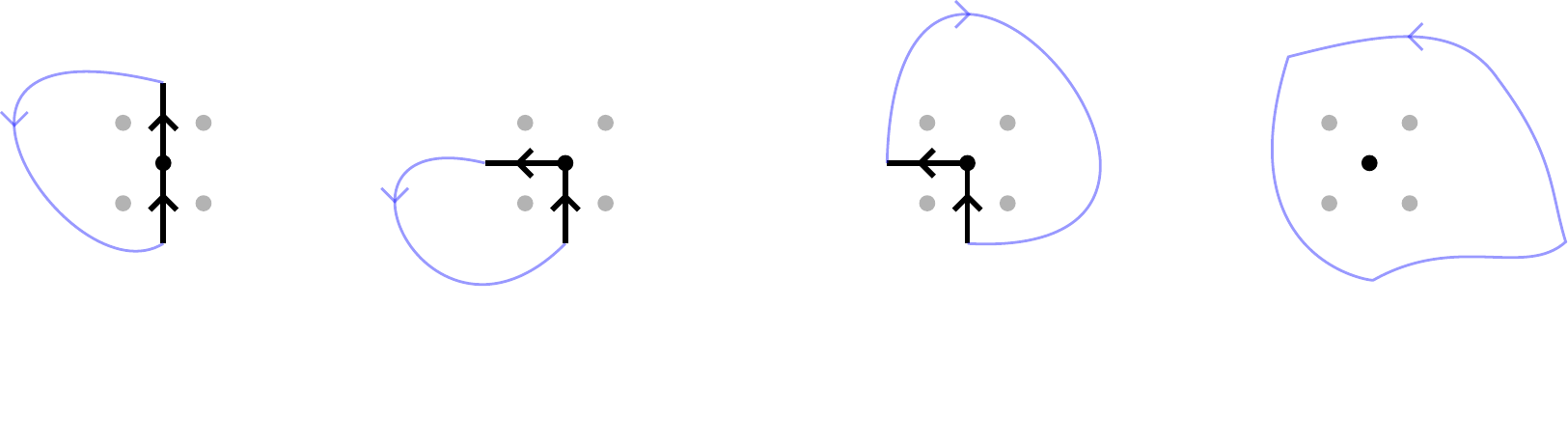%
\caption{Illustration of topological and metric weights. The points in $\neighbor_v$ are in grey.}%
\label{fig:topAndMetricWeights}%
\def\myScaleVar{1}
\end{figure}

If $s$ is the corresponding sock of a tiling $t$ and $\gamma$ is a cycle of $s$, then the angle $\angle(\gamma,v)$ of a vertex $v \in \gamma$ is the difference between the angle of the edge of $\gamma$ leaving $v$ and the angle of the edge of $\gamma$ entering $v$, counted in counterclockwise laps. In other words, a vertex $v$ where the curve goes straight has angle $0$, whereas a vertex where a left (resp. right) turn occurs has angle $1/4$ (resp. $-1/4$), as shown in Figure \ref{fig:angles}.
\begin{figure}[ht]%
\centering
\def\svgwidth{0.5\columnwidth}
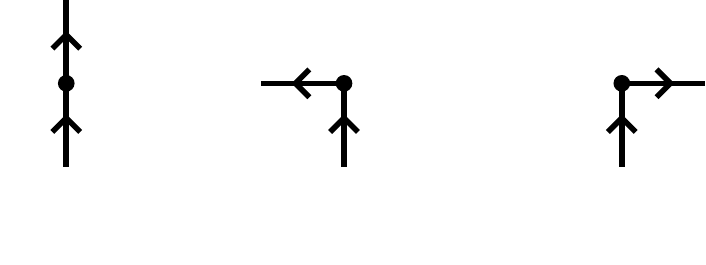
\caption{Illustration of the angle of a vertex.}%
\label{fig:angles}%
\end{figure}

The \emph{boundary charge} of a curve $\gamma$ is $\charge_{\partial}(\gamma) = \sum_{v \in \gamma} \angle(\gamma,v) \ccol(v)$. Notice that 
\begin{equation}
T^{\ez}(t) = \sum_{\gamma \scalebox{0.7}{\mbox{ cycle of }} s} \charge_{\partial}(\gamma).
\label{eq:ezpretwist}
\end{equation} 
We now set out to prove that, for each $\gamma$, $\charge_{\partial}(\gamma) = \charge_{\interior}(\gamma)$, which will complete the proof of Proposition \ref{prop:duplexes}.
\begin{lemma}
\label{lemma:metricColorZero}
For each cycle $\gamma$ of $s$, 
$$\sum_{v \in \ZZ^2} \metricw{\gamma}{v} \ccol(v) = 0.$$
\end{lemma}
\begin{proof}
$$\sum_{v \in \ZZ^2} \metricw{\gamma}{v} \ccol(v) = \sum_{u \in \ZZ^2 + (\half,\half)} \left(\frac{1}{4} \wind(\gamma,u) \sum_{v \in \neighbor_u} \ccol(v)\right) = 0.$$  
\end{proof}

\begin{lemma}
\label{lemma:topMinusMetric}
If $v \in \ZZ^2$, 
$$\topw{\gamma}{v} - \metricw{\gamma}{v} = \begin{cases} \angle(\gamma,v), &\text{if } v \in \gamma,\\
0, &\text{otherwise.} \end{cases} $$
\end{lemma}
\begin{proof}
Figure \ref{fig:topAndMetricWeights} illustrates most of the elements needed in this proof. If $v \notin \gamma$, then $\wind(\gamma,u) = \wind(\gamma,v)$ for each $u \in \neighbor_v$, so $\metricw{\gamma}{v} = \topw{\gamma}{v} = \wind(\gamma,v).$ 

If $v \in \gamma$ and the curve goes straight at $v$, then, for some $k$, two points in $\neighbor_v$ have winding number $k$ and the other two have winding number $k+1$, so $\metricw{\gamma}{v} = \frac{1}{4}(k + k + (k+1) + (k+1)) = \half(k + (k+1)) = \topw{\gamma}{v}.$

If $v \in \gamma$ and the curves turns left at $v$, then the winding numbers are $k,k,k,k+1$, so $\topw{\gamma}{v} - \metricw{\gamma}{v} = 1/4 = \angle(\gamma,v)$. Analogously, $\topw{\gamma}{v} - \metricw{\gamma}{v} = - 1/4 = \angle(\gamma,v)$ if $\gamma$ turns right at $v$.
\end{proof}

\begin{lemma}
\label{lemma:ptprimeEqualsTwist}
For each cycle $\gamma$ of $s$, $\charge_{\partial}(\gamma) = \charge_{\interior}(\gamma)$.
\end{lemma}
\begin{proof}
By Lemma \ref{lemma:topMinusMetric}, the boundary charge of $\gamma$ can also be written as 
\begin{align*}
\charge_{\partial}(\gamma) &= \sum_{v \in \gamma} \angle(\gamma,v) \ccol(v) \\
													 &= \sum_{v \in \ZZ^2} (\topw{\gamma}{v} - \metricw{\gamma}{v})\ccol(v)\\
                           &=\sum_{v \in \ZZ^2} \topw{\gamma}{v}\ccol(v) - \sum_{v \in \ZZ^2}\metricw{\gamma}{v}\ccol(v)\\
													&= \sum_{v \in \ZZ^2} \topw{\gamma}{v}\ccol(v) = \charge_{\interior}(\gamma),
\end{align*}
the fourth equality holding because of Lemma \ref{lemma:metricColorZero}; the last equality is Lemma \ref{lemma:interiorCharge}.
\end{proof}

\begin{proof}[Proof of Proposition \ref{prop:duplexes}]

Lemma \ref{lemma:ptprimeEqualsTwist} and Equation \eqref{eq:ezpretwist} imply that $P_t'(1) = T^{\ez}(t)$. Together with Lemma \ref{lemma:pretwistsNotAxis}, this proves the result.
\end{proof}

\begin{rem}
Proposition \ref{prop:duplexes} and Proposition 3.6 in \cite{segundoartigo} establish Remark 10.1 in \cite{primeiroartigo}. 
\end{rem}

\bibliography{biblio}{}
\bibliographystyle{plain}

\noindent
\footnotesize
Departamento de Matem\'atica, PUC-Rio \\
Rua Marqu\^es de S\~ao Vicente, 225, Rio de Janeiro, RJ 22451-900, Brazil \\
\url{milet@mat.puc-rio.br}\\
\url{saldanha@puc-rio.br}

\end{document}